\documentclass{article}

\usepackage{amsfonts}
\usepackage{amsthm}
\usepackage{amsmath}
\usepackage{amssymb}
\usepackage{mathtools} 
\usepackage{authblk} 
\usepackage{xcolor}
\usepackage{soul} 
\usepackage{mathrsfs}
\usepackage{graphicx}
\usepackage{caption}
\usepackage{subcaption}
\usepackage{float}
\usepackage{soul} 
\usepackage{mathrsfs}

\usepackage[top=1in, bottom=1in, left=1in, right=1in]{geometry}
\newcommand{\norm}[1]{\left \lVert#1\right \rVert}
\newtheorem{theorem}{Theorem}[section]
\newtheorem{lemma}{Lemma}[section]

\newtheorem{assumption}{Assumption}[section]
\theoremstyle{definition}
\newtheorem{definition}{Definition}[section]

\theoremstyle{remark}
\newtheorem{remark}{Remark}[section]

\definecolor{cucol}{rgb}{0,0,0.8}
\definecolor{afcol}{rgb}{1,0,0}

\numberwithin{equation}{section}

\begin{document}
	\title{Asymptotic stability analysis of Riemann-Liouville fractional stochastic neutral differential equations} 
	
	\author[]{Arzu Ahmadova \thanks{Email: \texttt{arzu.ahmadova@emu.edu.tr}}}
	\author[]{Nazim I. Mahmudov\thanks{ Corresponding author. Email: \texttt{nazim.mahmudov@emu.edu.tr}}}
	

	\affil{Department of Mathematics, Eastern Mediterranean University,  Famagusta, 99628
		T. R. Northern Cyprus}
	
	\date{}
	
		
	\maketitle

	\begin{abstract}
		The novelty of our paper is to establish results on asymptotic stability of mild solutions in $p$th moment to Riemann-Liouville fractional stochastic neutral differential equations (for short Riemann-Liouville FSNDEs) of order $\alpha \in (\frac{1}{2},1)$ using a Banach's contraction mapping principle. The core point of this paper is to derive the mild solution of FSNDEs involving Riemann-Liouville fractional time-derivative by applying the stochastic version of variation of constants formula. The results are obtained with the help of the theory of fractional differential equations, some properties of Mittag-Leffler functions and asymptotic analysis under the assumption that the corresponding fractional stochastic neutral dynamical system is asymptotically stable.

	\end{abstract}
	
	

	\section{Introduction}\label{Sec:intro}
	
	Over the years, many results have been investigated on the theory and applications of stochastic differential equations (SDEs) \cite{ito,oksendal,prato}. The deterministic models often oscillate due to noise. Certainly, the extension of these models is essential to consider stochastic models, in which the connected parameters are
	considered as appropriate Brownian motion and stochastic processes. The modeling of most problems in real-world problems is described by stochastic differential equations rather than deterministic equations. Thus, it is of great importance to design stochastic effects in the study of fractional-order dynamical systems. In particular, fractional stochastic differential equations (FSDEs) which are a generalization of differential equations by the use of fractional and stochastic calculus are more popular due to their applications in modeling and mathematical finance. Recently, FSDEs are intensively applied to model mathematical problems in finance \cite{farhadi,tien}, dynamics of complex systems in engineering \cite{wang} and other areas \cite{gangaram,mandelbrot}. Most of the results on fractional stochastic dynamical systems are limited to prove existence and uniqueness of mild solutions using fixed point theorem \cite{arzu,barbu,feng-li,mahmudov,rodkina,sakthivel,taniguchi}.

	Several outcomes have been investigated on the qualitative theory and applications of fractional stochastic functional differential equations (FSDEs) \cite{arzu-ismail-mahmudov,ismail,shen,balachandran}. For instance, Mahmudov and McKibben \cite{mahmudov-mckibben} studied the approximate controllability of fractional evolution equations involving generalized Riemann-Liouville fractional derivative in Hilbert spaces. Moreover, the stability theory for stochastic differential equations has been considered by means of exponential functions which lead to attain more efficient consequences in \cite{mao1}. Especially, Cao and Zhu \cite{cao} considered $p$th exponential stability of impulsive stochastic functional differential equations based on vector Lyapunov function. Based on average dwell-time method, Razumikhin-type technique and vector Lyapunov function, some novel stability method are obtained for impulsive  stochastic functional differential systems. On a study of different types of stability studies for FSDEs can be found in \cite{agarwal,ahmadova-mahmudov,tarif,Karaman,mao2,ragusa,wangj}.
	
	Results on the asymptotic behavior of solutions of fractional differential equations with Caputo and Rieamnn-Liouville fractional time-derivative are relatively scarce in the literature. Mahmudov \cite{mahmudov} derived an explicit solution formula to linear inhomogeneous delayed Langevin equation involving two Riemann-Liouville fractional derivatives and  studied existence and uniqueness, and Ulam-Hyers stability of solutions. In \cite{cong}, Cong et al. investigate the asymptotic behavior of solutions of the perturbed linear fractional differential system. Cong et al. \cite{stefan} proved the theorem of linearized asymptotic stability for fractional differential equations. More precisely, they showed that an equilibrium of a nonlinear Caputo fractional differential equation is asymptotically stable if its linearization at the equilibrium is asymptotically stable. There are only a few papers related to asymptotic stability of solutions of fractional stochastic differential equations which can be found in \cite{doan,sakthivel-ren}. Sakthivel et al. \cite{mahmudov-sakthivel} studied existence and asymptotic stability in $p$th moment of a mild solution to a class of nonlinear fractional neutral stochastic differential equations with infinite delays in Hilbert spaces. The same asymptotic stability in $p$th moment of a mild solutions of nonlinear impulsive stochastic differential equations and impulsive stochastic partial differential equations with infinite delays was discussed in \cite{luo1} and \cite{luo2}, respectively.
	
	To the best of our knowledge, the asymptotic stability of mild solutions for fractional stochastic neutral differential equations with Riemann-Liouville fractional derivative are an untreated topic in the present literature. Due to lack of asymptotic stability of mild solutions to Riemann-Liouville FSNDEs, this motivates us to establish new results on the asymptotic analysis of fractional stochastic differential equations with Riemann-Liouville fractional time-derivative involving matrix coefficients.
	
	Therefore, the plan of this paper is systematized as below: Section \ref{Sec:prel} is a preparatory section where we recall some basic notions and results from fractional calculus and fractional differential equations. Then we resort the setting for main results of the theory and we impose certain assumptions, definitions of stability and asymptotic stability in $p$th moment of mild solutions to Riemann-Liouville FSNDEs stochastic analysis. In Section \ref{stochastic}, first we verify the continuity of operator in $p$th moment on $[0,\infty)$. Then we show global existence and uniqueness of mild solution under various assumptions by a Banach's contraction mapping principle. Section \ref{asymptotic} is devoted to proving asymptotic stability of mild solutions to Riemann-Liouville FSNDEs. With the help of properties of Mittag-Leffler functions, we show that $\Psi$ is well-defined. Finally, we study asymptotic stability in $p$th of Riemann-Liouville FSNDEs of fractional-order $\alpha \in (\frac{1}{2},1)$. Section \ref{concl} is for the conclusion and future work by providing several open problems.
	
	To conclude this introductory section, we introduce the norm of the matrix which are used throughout the paper. For any constant matrix $A=(\,a_{ij})\,_{n\times n} \in \mathbb{R}^{n\times n}$, the norm of the matrix $A$, according to the maximum norm on $\mathbb{R}^{n}$ is
	\begin{equation*}
		\|A\|=\max_{1\leq i\leq  n}\sum_{j=1}^{n}|a_{ij}|.
	\end{equation*}
	Throughout this paper, we assume that $p\geq 2$ is an integer.
	
	\section{Mathematical preliminaries and statement of the main results}\label{Sec:prel}
	
	We embark on this section by briefly presenting some necessary facts and definitions from fractional calculus and special functions which are used throughout the paper.
	\subsection{Fractional Calculus}
	
	\begin{definition}[\cite{miller-ross}, \cite{podlubny}]
		The Riemann--Liouville integral operator of fractional order $\alpha>0$ is defined by
		\begin{equation*}
			\label{RLint:def}
			I^{\alpha}_{0+}g(t)=\frac{1}{\Gamma(\alpha)}\int_0^t(t-r)^{\alpha-1}g(r)\,\mathrm{d}r \\,\quad \text{for} \quad t>0.
		\end{equation*}
		The Riemann--Liouville derivative operator of fractional order $\alpha>0$ is defined by
		\begin{equation*}
			D^{\alpha}_{0+}g(t)=\frac{\mathrm{d}^n}{\mathrm{d}t^n}\left(I^{n-\alpha}_{0+}g(t)\right), \quad \text{where}\quad n-1<\alpha \leq n.
		\end{equation*}
	\end{definition}
	
	\begin{definition}[\cite{diethelm,podlubny,samko}]
		Suppose that $\alpha>0$, $t>0$. The Caputo derivative operator of fractional order $\alpha$ is defined by:
		\begin{equation*}
			\prescript{C}{}D^{\alpha}_{0+}g(t)=\prescript{}{}I^{n-\alpha}_{0+}\left(\frac{\mathrm{d}^n}{\mathrm{d}t^n}g(t)\right), \quad \text{where}\quad n-1<\alpha \leq  n.
		\end{equation*}
		In particular, for $\alpha \in (0,1)$
		\begin{equation*}
			\prescript{}{}I^{\alpha}_{0^{+}}\prescript{C}{}D^{\alpha}_{0^{+}}f(t)=f(t)-f(0).
		\end{equation*}
	\end{definition}
	The Riemann--Liouville fractional integral operator and the Caputo fractional derivative have the following property for $\alpha \geq 0$ \cite{kilbas,podlubny}:
	\begin{equation*} \prescript{}{}I^{\alpha}_{0+}(\prescript{C}{}D^{\alpha}_{0+}g(t))=g(t)-\sum_{k=0}^{n-1}\frac{t^{k}g^{(k)}(0)}{\Gamma(k+1)}.
	\end{equation*}
	The relationship between the Riemann--Liouville and Caputo fractional derivatives are as follows:
	\begin{equation*}
		\prescript{C}{}D^{\alpha}_{0+}g(t)=\prescript{}{}D^{\alpha}_{0+}g(t)-\sum_{k=0}^{n-1}\frac{t^{k-\alpha}g^{(k)}(0)}{\Gamma(k-\alpha+1)}, \quad \alpha\geq 0.
	\end{equation*}
	\begin{lemma}\label{RL}
		The relationship between Riemann-Liouville fractional derivative and integral as below:
		
		(i)\quad  If $f \in C(0,T]$, then for any point $t\in (0,T]$
		\begin{equation*}
			\prescript{RL}{}{D_{0+}^{\alpha}}\left(\prescript{}{}{I_{0+}^{\alpha}} f(t)\right) =f(t).
		\end{equation*}
		(ii)\quad If $f \in C(0,T]$ and $\prescript{}{}{I_{0+}^{1-\alpha}}f(t) \in C(0,T]$, then for any point $t\in (0,T]$
		\begin{align*}
			\prescript{}{}{I_{0+}^{\alpha}}\left( \prescript{RL}{}{D_{0+}^{\alpha}}f(t)\right) =f(t)-\frac{\prescript{}{}{I_{0+}^{1-\alpha}} f(t)|_{t=0}}{\Gamma(\alpha)}t^{\alpha-1}.
		\end{align*}
	\end{lemma}
	The Mittag-Leffler function is a generalization of the exponential function, first proposed as a single parameter function of one variable by using an infinite series \cite{1903}. Extensions to two or three parameters are well known and thoroughly studied in textbooks such as \cite{gorenflo}.
	\begin{definition} \label{Def:bML}
		We consider the matrix Mittag-Leffler functions with one and two parameters which are defined by
		\begin{align*}
			&E_{\alpha}(t^{\alpha}A)=\sum_{k=0}^{\infty}\frac{A^kt^{k\alpha}}{\Gamma(k\alpha+1)},\nonumber\\
			&E_{\alpha,\beta}(t^{\alpha}A)=\sum_{k=0}^{\infty}\frac{A^kt^{k\alpha}}{\Gamma(k\alpha+\beta)}.
		\end{align*}
	\end{definition}
	In through this paper, we define
	\begin{equation*}
		\Lambda^{s}_{\alpha}\coloneqq\{\lambda \in \mathbb{C}\backslash\{0\}: |\arg(\lambda)| > \frac{\alpha\pi}{2}\}.
	\end{equation*}
	For any matrix $A \in \mathbb{R}^{n\times n}$, the set $spec(A)$ is the spectrum of $A$, i.e
	\begin{equation*}
		spec(A)\coloneqq \left\lbrace \lambda \in \mathbb{C} : \lambda \quad \text{is an eigenvalue of the matrix} \quad A \right\rbrace .
	\end{equation*}
	It is well-known that the trivial solution of linear time-invariant fractional differential equation is asymptotically stable if and only if the spectrum $spec(A)$ of the matrix $A\in \mathbb{R}^{n\times n}$ satisfies the condition
	\begin{equation*}
		spec(A)\subset \Lambda^{s}_{\alpha}.
	\end{equation*}
	In our further consideration, we will use the following result.
	\begin{lemma}[\label{Lem21}\cite{htuan}]
		Let $A\in \mathbb{R}^{n\times n}$ and suppose that $spec(A)\subset \Lambda^{s}_{\alpha}$. Then the following statements are valid.
		
		\begin{itemize}
			\item There exists $t_{0}>0$ and a positive constant $\tilde{M}$ which depends on parameters $t_{0}, \alpha, A$ such that
			\begin{equation*}
				t^{\alpha-1}\|E_{\alpha,\alpha}(t^{\alpha}A)\|\leq \frac{\tilde{M}}{t^{\alpha+1}}, \quad \forall t\geq t_{0}.
			\end{equation*}
			\item The quantity
			\begin{equation*}
				t^{1-\alpha}\int_{0}^{t}(t-\tau)^{\alpha-1}\|E_{\alpha,\alpha}((t-\tau)^{\alpha}A)\|\tau^{\alpha-1}\mathrm{d}\tau
			\end{equation*}
			is bounded on $[0,\infty)$, i.e.
			\begin{equation*}
				\sup_{t\geq 0} t^{1-\alpha}\int_{0}^{t}(t-\tau)^{\alpha-1}\|E_{\alpha,\alpha}((t-\tau)^{\alpha}A)\|\tau^{\alpha-1}\mathrm{d}\tau<\infty.
			\end{equation*}
		\end{itemize}
	\end{lemma}
	\textbf{Variation of constants method} The first step in our approach is to use the fractionally modified version of the variation of constants method which is established in \cite[Theorem 7.2 and Remark 7.1]{diethelm} in order to obtain the solution of \eqref{fstoc} by applying stochastic version of variation of constants formula in Section \ref{stochastic}. Namely, the solution of a single inhomogeneous Riemann-Liouville fractional differential equation of the form
	\begin{align}\label{RLfde}
		\begin{cases}
			\prescript{RL}{}{D_{0+}^{\alpha}}y(t)=A y(t)+h(t), \quad t>0;\\ \prescript{}{}{I_{0+}^{k-\alpha}}y(t)|_{t=0}=b_{k},\quad k=1,2,\ldots,n,
		\end{cases}
	\end{align}
	where $n-1<\alpha\leq n$ and $A\in \mathbb{R}^{n\times n}$ constant matrix. The problem \eqref{RLfde} was analytically solved in \cite{samko} by successive iteration method. However, with the help of Laplace transform, one can obtain the same solution directly and more quickly. Indeed, taking into account initial conditions and using the inverse Laplace transform, the problem \eqref{RLfde} can be expressed in the form:
	\begin{align*}
		y(t)=\sum_{k=1}^{n}b_{k}t^{\alpha-k}E_{\alpha,\alpha-k+1}(A t^{\alpha})+\int_{0}^{t}(t-\tau)^{\alpha-1}E_{\alpha,\alpha}(A (t-\tau)^{\alpha})h(\tau)\mathrm{d}\tau.
	\end{align*}
	The following results play a necessary role on proofs of main results in Section \ref{stochastic}.
	
	Suppose that $\alpha_1$,$\alpha_2 >1$ and $\frac{1}{\alpha_1}+\frac{1}{\alpha_2}=1$. If $|f(t)|^{\alpha_1}$, $|g(t)|^{\alpha_2}\in L^{1}(\Omega)$, then $|f(t)g(t)|\in L^{1}(\Omega)$ and
	\begin{equation}\label{holder}
		\int_{\Omega}|f(t)g(t)|\mathrm{d}t \leq \left( \int_{\Omega}|f(t)|^{\alpha_1}\mathrm{d}t\right)^{\frac{1}{\alpha_1}} \left(\int_{\Omega}|g(t)|^{\alpha_2}\mathrm{d}t\right)^{\frac{1}{\alpha_2}},
	\end{equation}
	where $L^{1}(\Omega)$ represents the Banach space of all Lebesgue measurable functions $f:\Omega \to \mathbb{R}$ with $\int_{\Omega}|f(t)|\mathrm{d}t < \infty$. Especially, when $\alpha_1=\alpha_2=2$, the inequality \eqref{holder} reduces to the Cauchy-Schwartz  inequality
	\begin{equation}\label{cauchy}
		\left( \int_{\Omega}|f(t)g(t)|\mathrm{d}t \right) ^{2}\leq \int_{\Omega}|f(t)|^{2}\mathrm{d}t  \int_{\Omega}|g(t)|^{2}\mathrm{d}t.
	\end{equation}
	Let $n \in \mathbb{N}$, $q>1$ and $x_{i}\in\mathbb{R}_{+}$, $i=1,2,\ldots,n$. Then, the following inequality holds true
	\begin{equation*}
		\|\sum_{i=1}^{n}x_{i}\|^{q} \leq n^{q-1}\sum_{i=1}^{n}\|x_{i}\|^{q}.
	\end{equation*}
	In particular, when $q=2$, we use the following inequality within the estimations in this paper:
	\begin{equation}\label{ineqq}
		\|\sum_{i=1}^{n}x_{i}\|^{2} \leq n\sum_{i=1}^{n}\|x_{i}\|^{2}.
	\end{equation}
	
	\subsection{Statement of main results}
	Let $\mathbb{R}^{n}$ be endowed with the standard Euclidean norm $\|\cdot\|$. For each $t \geq [0,T]$ let $\Xi_{t} \coloneqq \mathbb{L}^{p}(\Omega, \mathscr{F}_{T},\mathbb{P})$ denote the space of all $\mathscr{F}_{t}$ measurable integrable random variables  with values in $\mathbb{R}^{n}$.
	
	We consider Riemann-Liouville fractional stochastic neutral differential equations of order $\alpha \in (\frac{1}{2},1)$ has the following form on $[0,T]$:
	\begin{equation}\label{fstoc}
		\begin{cases}
			\prescript{}{}D_{0^{+}}^{\alpha}\left( X(t) +g(t,X(t))\right)=A  X(t)+b(t,X(t))+\sigma(t, X(t))\frac{\mathrm{d}W_t}{\mathrm{d}t} ,\\
			\prescript{}{}{I_{0+}^{1-\alpha}}\left( X(t)+g(t,X(t))\right)|_{t=0}=\rho.
		\end{cases}
	\end{equation}
	The coefficients $g,b,\sigma : [0,T] \times \mathbb{R}^{n} \to \mathbb{R}^{n}$ are measurable and bounded functions, $A\in \mathbb{R}^{n\times n}$ and the initial condition $\rho$ in integral form is an $\mathscr{F}_{0}$-measurable. Let $( W_{t})_{t\geq 0} $ denote a standard scalar Brownian motion on a complete probability space $(\Omega,\mathscr{F}, \mathbb{F}, \mathbb{P})$ with filtration $\mathbb{F} \coloneqq \left\lbrace \mathscr{F}_{t}\right\rbrace_{t \geq 0}$.
	
	Let $H$ be the Banach space of all $\mathbb{R}^{n}$-valued  $\mathbb{F}_{t}$-adapted process $\xi$ satisfying the following norm defined by
	
	\begin{align*}
		\| \xi\|^{p}_{H} \coloneqq \sup_{t\geq 0} \textbf{E}\|t^{1-\alpha}\xi(t)\|^{p}< \infty.
	\end{align*}
	
	Let us recall the definition of mild solution to Riemann-Liouville FSNDEs \eqref{fstoc}.
	\begin{definition}
		A stochastic process $ \left\lbrace X(t), t\in [0,T] \right\rbrace$ is called  a mild solution of \eqref{fstoc} if
		\begin{itemize}
			\item $X(t)$ is adapted to $\left\lbrace \mathscr{F}_{t}\right\rbrace _{t \geq 0}$ with
			$ \int_{0}^{t}  \|X(t)\|^{p}_{H}\mathrm{d}t< \infty$ almost everywhere;
			\item $X(t)\in H$ has continuous path on $t\in [0,T]$ a.s. and satisfies the following Volterra integral equation of second kind for each $t\in [0,T]$:
		\end{itemize}	
	\end{definition}	
	\allowdisplaybreaks
	\begin{align}  \label{integral equation}
		X(t)=\frac{t^{\alpha-1}\rho}{\Gamma(\alpha)}-g(t,X(t))&+\frac{1}{\Gamma(\alpha)}\int_{0}^{t}A(t-\tau)^{\alpha-1}g(\tau,X(\tau))\mathrm{d}\tau\nonumber\nonumber\\
		&+\frac{1}{\Gamma(\alpha)}\int_{0}^{t}(t-\tau)^{\alpha-1}b(\tau,X(\tau))\mathrm{d}\tau\nonumber\\
		&+\frac{1}{\Gamma(\alpha)}\int_{0}^{t}(t-\tau)^{\alpha-1}\sigma(\tau,X(\tau))\mathrm{d}W_\tau.
	\end{align}
	Applying variation of constants formula for deterministic Riemann-Liouville fractional differential equations and then adapting this formula for stochastic case, for each $\rho\in \mathbb{R}^{n}$, we obtain
	\begin{align}\label{RLFstoc}
		X(t)&= t^{\alpha-1}E_{\alpha,\alpha}(t^{\alpha}A)\rho - g(t,X(t)) \nonumber\\
		&-\int_{0}^{t} A(t-\tau)^{\alpha-1}E_{\alpha,\alpha}((t-\tau)^{\alpha}A) g(\tau,X(\tau))\mathrm{d}\tau\nonumber \\
		&+\int_{0}^{t} (t-\tau)^{\alpha-1}E_{\alpha,\alpha}((t-\tau)^{\alpha}A)b(\tau,X(\tau))\mathrm{d}\tau \nonumber\\
		&+\int_{0}^{t} (t-\tau)^{\alpha-1}E_{\alpha,\alpha}((t-\tau)^{\alpha}A)\sigma(\tau,X(\tau))\mathrm{d}W(\tau).
	\end{align}
	\begin{theorem}[Coincidence between the notion of integral equation and mild solution]
		The integral equation \eqref{integral equation} and the mild solution \eqref{RLfde} of Riemann-Liouville FSNDEs coincide each other.
	\end{theorem}
	\begin{proof}
		Using martingale representation and It\^{o}'s isometry, this theorem can be proved in a similar way shown in \cite{arzu} by assuming $p\geq 2$.
	\end{proof}
	We impose the following assumptions on \eqref{fstoc}:
	\begin{assumption}\label{A0}
		Let $A \in \mathbb{R}^{n\times n}$ be constant matrix and suppose that $spec(A)\subset \Lambda^{s}_{\alpha}$.
	\end{assumption}
	
	\begin{assumption} \label{A1}
		The coefficients $g, b,\sigma$ satisfy global Lipschitz continuity: there exists $L_{g},L_{b},L_{\sigma}>0$ such that for all $x,y\in \mathbb{R}^{n}$ , $t\geq 0$,
		\begin{align*}
			&\|b(t,x)-b(t,y)\| \leq L_{b}\|x-y\|, \quad \|g(t,x)-g(t,y)\| \leq L_{g}\|x-y\|, \\
			& \|\sigma(t,x)-\sigma(t,y)\| \leq L_{\sigma}\|x-y\|.
		\end{align*}
	\end{assumption}
	\begin{assumption}\label{A11}
		$b(\cdot,0)$ is $\mathbb{L}^{p}$ integrable i.e.
		\begin{equation*}
			\int_{0}^{T} \|b(\tau,0)\|^{p}\mathrm{d}\tau <\infty,
		\end{equation*}
		and $\sigma(\cdot,0)$ is essentially bounded i.e.
		\begin{equation*}
			ess\sup\limits_{\tau\in [0,T]}\|\sigma(\tau,0)\|< \infty.
		\end{equation*}
	\end{assumption}
	\begin{assumption}\label{A3}
		We impose the following condition on measurable functions $g,b,\sigma$ which are used on certain estimations in Section \ref{stochastic}.
		\begin{align*}
			g(t,0)=0, \quad b(t,0)=0, \quad \sigma(t,0)=0.
		\end{align*}
	\end{assumption}
	
	\begin{assumption}\label{A2}
		The following condition will be used in Section \ref{stochastic} to derive asymptotic stability results.
		\begin{align*}
			\Theta &\coloneqq 4^{p-1}\Big( L_{g}^{p}\|A\|^{p}M^{p}\left(B\left(\frac{p\alpha-1}{p-1},\frac{p\alpha-1}{p-1}\right)\right)^{p-1} T^{p\alpha-1}\\
			&+L_{b}^{p}M^{p}\left(B\left(\frac{p\alpha-1}{p-1},\frac{p\alpha-1}{p-1}\right)\right)^{p-1}T^{p\alpha-1}\\
			&+C_{p}L_{\sigma}^{p}M^{p}T^{p(\alpha-1)+\frac{p}{2}}\left( B(2\alpha-1,2\alpha-1)\right)^{p/2}\Big)<1,
		\end{align*}
		where $C_{p}=\left( \frac{p(p-1)}{2}\right)^{p/2}$, $M\coloneqq \sup_{t\in [0,T]}\|E_{\alpha,\alpha}(t^{\alpha}A)\|$ and $B$ is the beta function.
	\end{assumption}

	\begin{definition}\label{stability}
		Let $p\geq 2$ be an integer. The equation \ref{fstoc} is said to be stable in $p$th moment if for arbitrarily given $\epsilon>0$ there exists a $\delta>0$ such that
		\begin{equation}\label{stab}
			\sup_{t\geq 0}\textbf{E}\left(\|X(t)\|^{p}\right) <\epsilon, \quad \text{when}\quad \|\rho\|<\delta.
		\end{equation}
	\end{definition}
	\begin{definition}\label{asymptoticstab}
		Let $p\geq 2$ be an integer. The equation \ref{fstoc} is said to be asymptotically stable in $p$th moment if it is stable in $p$th moment and for any $\rho \in \mathbb{R}^{n}$ it holds
		\begin{equation}\label{astab}
			\lim\limits_{t\to\infty}\sup_{t\geq T}\textbf{E}\left(\|X(t)\|^{p}\right) =0.
		\end{equation}
	\end{definition}
	Now let us state the following well-known lemma (Da Prato and Zabczyk, 1992) which will be used in the sequel of the proofs of the main results.
	\begin{lemma} \label{prato}
		Let $p\geq 2, t>0$ and let $\Phi$ be an $\mathbb{R}^{n}$-valued predictable process such that $\textbf{E}\int\limits_{0}^{t}\|\Phi(\tau)\|^{p}\mathrm{d}\tau < \infty$. Then,
		\begin{align*}
			&\sup_{\tau\in [0,T]}\textbf{E}\norm{\int_{0}^{\tau}\Phi(u)\mathrm{d}W(u)}^{p}\\
			&\leq \left( \frac{p(p-1)}{2}\right)^{p/2}\left(\int_{0}^{t}\left(\textbf{E}\|\Phi(\tau)\|^{p} \right)^{2/p} \mathrm{d}\tau\right)^{p/2} .
		\end{align*}
	\end{lemma}

	\section{Existence and uniqueness of Riemann-Liouville fractional stochastic neutral differential equations}\label{stochastic}
	In this section, to show global existence and uniqueness results for Riemann-Liouville FSNDEs \eqref{fstoc}, we state and prove the following theorem.
	\begin{theorem}
		Suppose that the Assumptions \ref{A1}-\ref{A3} hold. Then for any $\rho\in \mathbb{R}^{n}$, there exists a unique mild solution of \eqref{fstoc}.
	\end{theorem}
	\begin{proof}
		To prove the asymptotic stability it is sufficient to show that the operator $\Psi$ has a fixed point in $H$. To show this result, we use the contraction mapping principle.
		
		Associated with the fixed point problem, we define an operator $\Psi: H\to H$ for $t \in [0,T]$ by
		\begin{align}\label{T.oper}
			(\Psi X)(t)&= t^{\alpha-1}E_{\alpha,\alpha}(t^{\alpha}A)\rho - g(t,X(t)) \nonumber\\
			&-\int_{0}^{t} A(t-\tau)^{\alpha-1}E_{\alpha,\alpha}((t-\tau)^{\alpha}A) g(\tau,X(\tau))\mathrm{d}\tau\nonumber \\
			&+\int_{0}^{t} (t-\tau)^{\alpha-1}E_{\alpha,\alpha}((t-\tau)^{\alpha}A)b(\tau,X(\tau))\mathrm{d}\tau \nonumber\\
			&+\int_{0}^{t} (t-\tau)^{\alpha-1}E_{\alpha,\alpha}((t-\tau)^{\alpha}A)\sigma(\tau,X(\tau))\mathrm{d}W(\tau).
		\end{align}
		
		First, we verify mean square continuity of $\Psi$ on $[0,T]$. Let $X\in H$, $t_{1}\geq 0$ and $r$ be sufficiently small and show that
		\begin{align}\label{final}
			&\textbf{E}\|(t_{1}+r)^{1-\alpha}(\Psi X)(t_{1}+r)-t_{1}^{1-\alpha}(\Psi X)(t_{1})\|^{p}\nonumber\\
			\leq& 5^{p-1}\sum_{i=1}^{5}\textbf{E}\|(t_{1}+r)^{1-\alpha}\mathcal{I}_{i}(t_{1}+r)-t_{1}^{\alpha}\mathcal{I}_{i}(t_{1})\|^{p}.
		\end{align}
		Note that
		\begin{align}\label{I1}
			&\textbf{E}\|(t_{1}+r)^{1-\alpha}\mathcal{I}_{1}(t_{1}+r)-t_{1}^{1-\alpha}\mathcal{I}_{1}(t_{1})\|^{p}\nonumber\\
			=&\textbf{E}\|\left[(t_{1}+r)^{1-\alpha}(t_{1}+r)^{\alpha-1} E_{\alpha,\alpha}((t_{1}+r)^{\alpha}A)-t_{1}^{1-\alpha}t_{1}^{\alpha-1}E_{\alpha,\alpha}(t_{1}^{\alpha}A)\right]\rho \|^{p}\nonumber\\
			=&\textbf{E}\|\left[ E_{\alpha,\alpha}((t_{1}+r)^{\alpha}A)-E_{\alpha,\alpha}(t_{1}^{\alpha}A)\right]\rho \|^{p}.
		\end{align}
		The strong continuity of $E_{\alpha,\alpha}(t^{\alpha}A)$ implies that the right-hand-side of \eqref{I1} goes to zero as $r \to 0$.
		
		By Assumptions \ref{A1} and \ref{A3}, we get
		\begin{align}\label{I2}
			&\textbf{E}\|(t_{1}+r)^{1-\alpha}\mathcal{I}_{2}(t_{1}+r)-t_{1}^{1-\alpha}\mathcal{I}_{2}(t_{1})\|^{p}\nonumber\\
			=&\textbf{E}\|(t_{1}+r)^{1-\alpha}g(t_{1}+r,X(t_{1}+r))-t_{1}^{1-\alpha}g(t_{1},X(t_{1}))\|^{p} \nonumber\\
			\leq&2^{p-1}\textbf{E}\|(t_{1}+r)^{1-\alpha}g(t_{1}+r,X(t_{1}+r))-(t_{1}+r)^{1-\alpha}g(t_{1}+r,X(t_{1}))\|^{p}\nonumber\\
			+&2^{p-1}\textbf{E}\|(t_{1}+r)^{1-\alpha}g(t_{1}+r,X(t_{1}))-t_{1}^{1-\alpha}g(t_{1},X(t_{1}))\|^{p}\nonumber\\
			\leq & 2^{p-1}L_{g}^{p}\textbf{E}\|(t_{1}+r)^{1-\alpha}\left( X(t_{1}+r)-X(t_{1})\right)\|^{p} \nonumber\\
			+&2^{p-1}\textbf{E}\|(t_{1}+r)^{1-\alpha}g(t_{1}+r,X(t_{1}))-t_{1}^{1-\alpha}g(t_{1},X(t_{1}))\|^{p}\to 0, \quad \text{as}\quad r \to 0.
		\end{align}
		In view of the H\"{o}lder's inequality and Assumptions  \ref{A1} and \ref{A3}, we have
		
		\allowdisplaybreaks
		{\small
			\begin{align}\label{I3}
				&\textbf{E}\|(t_{1}+r)^{1-\alpha}\mathcal{I}_{3}(t_{1}+r)-t_{1}^{1-\alpha}\mathcal{I}_{3}(t_{1})\|^{p}\nonumber\\
				&=\textbf{E}\|\int_{0}^{t_{1}+r} A(t_{1}+r)^{1-\alpha}(t_{1}+r-\tau)^{\alpha-1}E_{\alpha,\alpha}((t_{1}+r-\tau)^{\alpha}A) g(\tau,X(\tau))\mathrm{d}\tau \nonumber \\
				&-\int_{0}^{t_{1}} At_{1}^{1-\alpha}(t_{1}-\tau)^{\alpha-1}E_{\alpha,\alpha}((t_{1}-\tau)^{\alpha}A) g(\tau,X(\tau))\mathrm{d}\tau\|^{p}\nonumber\\
				&\leq 2^{p-1}\textbf{E}\|\int_{0}^{t_{1}+r} A\Big[(t_{1}+r)^{1-\alpha} (t_{1}+r-\tau)^{\alpha-1}E_{\alpha,\alpha}((t_{1}+r-\tau)^{\alpha}A)\nonumber\\
				&-t_{1}^{1-\alpha}(t_{1}-\tau)^{\alpha-1}E_{\alpha,\alpha}((t_{1}-\tau)^{\alpha}A)\Big] g(\tau,X(\tau))\mathrm{d}\tau\|^{p} \nonumber\\
				&+2^{p-1}\textbf{E}\norm{\int_{t_{1}}^{t_{1}+r} At_{1}^{1-\alpha}(t_{1}-\tau)^{\alpha-1} E_{\alpha,\alpha}((t_{1}-\tau)^{\alpha}A) g(\tau,X(\tau))\mathrm{d}\tau}^{p} \nonumber\\
				&\leq 2^{p-1}\textbf{E}\Big(\int_{0}^{t_{1}+r} \|A\|\Big[ \tau^{\alpha-1} (t_{1}+r)^{1-\alpha} (t_{1}+r-\tau)^{\alpha-1}E_{\alpha,\alpha}((t_{1}+r-\tau)^{\alpha}A)\nonumber\\
				&-\tau^{\alpha-1}t_{1}^{1-\alpha}(t_{1}-\tau)^{\alpha-1}E_{\alpha,\alpha}((t_{1}-\tau)^{\alpha}A)\Big]  \|\tau^{1-\alpha}g(\tau,X(\tau))\|\mathrm{d}\tau\Big)^{p} \nonumber\\
				&+2^{p-1}\textbf{E}\Big( \int_{t_{1}}^{t_{1}+r} \|A\| t_{1}^{1-\alpha}\tau^{\alpha-1}(t_{1}-\tau)^{\alpha-1}\| E_{\alpha,\alpha}((t_{1}-\tau)^{\alpha}A)\| \|\tau^{1-\alpha} g(\tau,X(\tau))\|\mathrm{d}\tau\Big)^{p} \nonumber\\
				&\leq 2^{p-1}L^{p}_{g}\|A\|^{p}\Big( \int_{0}^{t_{1}+r} \|\tau^{\alpha-1} (t_{1}+r)^{1-\alpha} (t_{1}+r-\tau)^{\alpha-1}E_{\alpha,\alpha}((t_{1}+r-\tau)^{\alpha}A)\nonumber\\
				&-\tau^{\alpha-1}t_{1}^{1-\alpha}(t_{1}-\tau)^{\alpha-1}E_{\alpha,\alpha}((t_{1}-\tau)^{\alpha}A)\|^{\frac{p}{p-1}}\mathrm{d}\tau\Big)^{p-1}\int_{0}^{t_{1}+r} \textbf{E} \|\tau^{1-\alpha}X(\tau)\|^{p} \mathrm{d}\tau\nonumber\\
				&+2^{p-1}L^{p}_{g}\|A\|^{p}\Big( \int_{t_{1}}^{t_{1}+r} t_{1}^{\frac{p(1-\alpha)}{p-1}}((t_{1}-\tau)\tau)^{\frac{p(\alpha-1)}{p-1}}\|E_{\alpha,\alpha}((t_{1}-\tau)^{\alpha}A)\|^{\frac{p}{p-1}}\mathrm{d}\tau\Big)^{p-1}\int_{t_{1}}^{t_{1}+r}\textbf{E}\|\tau^{1-\alpha}X(\tau)\|^{p} \mathrm{d}\tau.
			\end{align}
		}
		Thus, the right-hand-side of above inequality tends to zero as $r \to 0$.
		
		Next, we consider
		\allowdisplaybreaks
		{\small
			\begin{align}\label{I4}
				&\textbf{E}\|(t_{1}+r)^{1-\alpha}\mathcal{I}_{4}(t_{1}+r)-t_{1}^{1-\alpha}\mathcal{I}_{4}(t_{1})\|^{p}\nonumber\\
				&=\textbf{E}\|\int_{0}^{t_{1}+r} (t_{1}+r)^{1-\alpha} (t_{1}+r-\tau)^{\alpha-1}E_{\alpha,\alpha}((t_{1}+r-\tau)^{\alpha}A) b(\tau,X(\tau))\mathrm{d}\tau \nonumber\\
				&-\int_{0}^{t_{1}} t_{1}^{\alpha} (t_{1}-\tau)^{\alpha-1}E_{\alpha,\alpha}((t_{1}-\tau)^{\alpha}A) b(\tau,X(\tau))\mathrm{d}\tau\|^{p}\nonumber\\
				&\leq 2^{p-1}\textbf{E}\|\int_{0}^{t_{1}+r} \Big[(t_{1}+r)^{1-\alpha} (t_{1}+r-\tau)^{\alpha-1}E_{\alpha,\alpha}((t_{1}+r-\tau)^{\alpha}A)\nonumber\\
				&-t_{1}^{1-\alpha}(t_{1}-\tau)^{\alpha-1}E_{\alpha,\alpha}((t_{1}-\tau)^{\alpha}A)\Big] b(\tau,X(\tau))\mathrm{d}\tau\|^{p} \nonumber\\
				&+2^{p-1}\textbf{E}\norm{\int_{t_{1}}^{t_{1}+r} t_{1}^{1-\alpha}(t_{1}-\tau)^{\alpha-1} E_{\alpha,\alpha}((t_{1}-\tau)^{\alpha}A) b(\tau,X(\tau))\mathrm{d}\tau}^{p} \nonumber\\
				&\leq 2^{p-1}\textbf{E}\Big( \int_{0}^{t_{1}+r}\|\tau^{\alpha-1} (t_{1}+r)^{1-\alpha} (t_{1}+r-\tau)^{\alpha-1}E_{\alpha,\alpha}((t_{1}+r-\tau)^{\alpha}A)\nonumber\\
				&-\tau^{\alpha-1}t_{1}^{1-\alpha}(t_{1}-\tau)^{\alpha-1}E_{\alpha,\alpha}((t_{1}-\tau)^{\alpha}A)\| \|\tau^{1-\alpha}b(\tau,X(\tau))\|\mathrm{d}\tau\Big)^{p} \nonumber\\
				&+2^{p-1}\textbf{E}\left( \int_{t_{1}}^{t_{1}+r} \tau^{\alpha-1}(t_{1}-\tau)^{\alpha-1}\| E_{\alpha,\alpha}((t_{1}-\tau)^{\alpha}A)\| \|\tau^{1-\alpha} b(\tau,X(\tau))\|\mathrm{d}\tau\right)^{p} \nonumber\\
				&\leq 2^{p-1}L^{p}_{g}\Big(\int_{0}^{t_{1}+r} \|\tau^{\alpha-1} (t_{1}+r)^{1-\alpha} (t_{1}+r-\tau)^{\alpha-1}E_{\alpha,\alpha}((t_{1}+r-\tau)^{\alpha}A)\nonumber\\
				&-\tau^{\alpha-1}t_{1}^{1-\alpha}(t_{1}-\tau)^{\alpha-1}E_{\alpha,\alpha}((t_{1}-\tau)^{\alpha}A)\|^{\frac{p}{p-1}}\mathrm{d}\tau\Big)^{p-1} \int_{0}^{t_{1}+r} \textbf{E} \|\tau^{1-\alpha}X(\tau)\|^{p} \mathrm{d}\tau\nonumber\\
				&+2^{p-1}L^{p}_{g}\left( \int_{t_{1}}^{t_{1}+r} ((t_{1}-\tau)\tau)^{\frac{p(\alpha-1)}{p-1}}\| E_{\alpha,\alpha}((t_{1}-\tau)^{\alpha}A)\|^{\frac{p}{p-1}}\mathrm{d}\tau\right)^{p-1}\int_{t_{1}}^{t_{1}+r}\textbf{E}\|\tau^{1-\alpha}X(\tau)\|^{p} \mathrm{d}\tau.
			\end{align}
		}
		Therefore, the right-hand-side  of the above inequality tends to zero as $r\to 0$ and $\epsilon$ is sufficiently small.
		
		Further, we have
		\allowdisplaybreaks
		{\small
			\begin{align}\label{I5}
				&\textbf{E}\|(t_{1}+r)^{1-\alpha}\mathcal{I}_{5}(t_{1}+r)-t_{1}^{1-\alpha}\mathcal{I}_{5}(t_{1})\|^{p}\nonumber\\
				&=\textbf{E}\|\int_{0}^{t_{1}+r}(t_{1}+r)^{1-\alpha} (t_{1}+r-\tau)^{\alpha-1}E_{\alpha,\alpha}((t_{1}+r-\tau)^{\alpha}A) \sigma(\tau,X(\tau))\mathrm{d}W(\tau)\nonumber\\
				&-\int_{0}^{t_{1}}t_{1}^{1-\alpha} (t_{1}-\tau)^{\alpha-1}E_{\alpha,\alpha}((t_{1}-\tau)^{\alpha}A) \sigma(\tau,X(\tau))\mathrm{d}W(\tau)\|^{p}\nonumber\\
				&\leq 2^{p-1}\textbf{E}\|\int_{0}^{t_{1}+r} \Big[(t_{1}+r)^{1-\alpha} (t_{1}+r-\tau)^{\alpha-1}E_{\alpha,\alpha}((t_{1}+r-\tau)^{\alpha}A)\nonumber\\
				&-t_{1}^{1-\alpha}(t_{1}-\tau)^{\alpha-1}E_{\alpha,\alpha}((t_{1}-\tau)^{\alpha}A)\Big] \sigma(\tau,X(\tau))\mathrm{d}W(\tau) \|^{p} \nonumber\\
				&+2^{p-1}\textbf{E}\norm{\int_{t_{1}}^{t_{1}+r} t_{1}^{1-\alpha}(t_{1}-\tau)^{\alpha-1} E_{\alpha,\alpha}((t_{1}-\tau)^{\alpha}A) \sigma(\tau,X(\tau))\mathrm{d}W(\tau)}^{p} \nonumber\\
				&\leq 2^{p-1}C_{p}\textbf{E}\Big( \int_{0}^{t_{1}+r} \Big[\tau^{\alpha-1}(t_{1}+r)^{1-\alpha} (t_{1}+r-\tau)^{\alpha-1}E_{\alpha,\alpha}((t_{1}+r-\tau)^{\alpha}A)\nonumber\\
				&-\tau^{\alpha-1}t_{1}^{1-\alpha}(t_{1}-\tau)^{\alpha-1}E_{\alpha,\alpha}((t_{1}-\tau)^{\alpha}A)\Big]^{2} \|\tau^{1-\alpha}\sigma(\tau,X(\tau))\|^{2}\mathrm{d}\tau\Big) ^{p/2} \nonumber\\
				&+2^{p-1}C_{p}\textbf{E}\left( \int_{t_{1}}^{t_{1}+r} \tau^{2\alpha-2}(t_{1}-\tau)^{2\alpha-2} \|E_{\alpha,\alpha}((t_{1}-\tau)^{\alpha}A)\|^{2} \|\tau^{1-\alpha}\sigma(\tau,X(\tau))\|^{2}\mathrm{d}\tau \right) ^{p/2} \nonumber\\
				&\leq 2^{p-1}C_{p}L^{p}_{\sigma}\Big( \int_{0}^{t_{1}+r} \Big[\tau^{\alpha-1}(t_{1}+r)^{1-\alpha} (t_{1}+r-\tau)^{\alpha-1}E_{\alpha,\alpha}((t_{1}+r-\tau)^{\alpha}A)\nonumber\\
				&-\tau^{\alpha-1}t_{1}^{1-\alpha}(t_{1}-\tau)^{\alpha-1}E_{\alpha,\alpha}((t_{1}-\tau)^{\alpha}A)\Big]^{2} \textbf{E}\|\tau^{1-\alpha}X(\tau)\|^{2}\mathrm{d}\tau\Big)^{p/2} \nonumber\\
				&+2^{p-1}C_{p}L^{p}_{\sigma}\left( \int_{t_{1}}^{t_{1}+r}\tau^{2\alpha-2} (t_{1}-\tau)^{2\alpha-2} \|E_{\alpha,\alpha}((t_{1}-\tau)^{\alpha}A)\|^{2}\textbf{E} \|\tau^{1-\alpha}X(\tau)\|^{2}\mathrm{d}\tau \right)^{p/2}\nonumber\\
				&\leq 2^{p-1}C_{p}L^{p}_{\sigma}\Big(\int_{0}^{t_{1}+r} \Big[\tau^{\alpha-1}(t_{1}+r)^{1-\alpha} (t_{1}+r-\tau)^{\alpha-1}E_{\alpha,\alpha}((t_{1}+r-\tau)^{\alpha}A)\nonumber\\
				&-\tau^{\alpha-1}t_{1}^{1-\alpha}(t_{1}-\tau)^{\alpha-1}E_{\alpha,\alpha}((t_{1}-\tau)^{\alpha}A)\Big]^{\frac{2p}{p-2}} \mathrm{d}\tau\Big)^{\frac{p-2}{2}} \int_{0}^{t_{1}+r}\textbf{E}\|\tau^{1-\alpha}X(\tau)\|^{p}\mathrm{d}\tau\nonumber\\
				&+2^{p-1}C_{p}L^{p}_{\sigma} \Big( \int_{t_{1}}^{t_{1}+r}\left(  (t_{1}-\tau)^{2\alpha-2} \|E_{\alpha,\alpha}((t_{1}-\tau)^{\alpha}A)\|^{2}\right)^{\frac{p}{p-2}} \mathrm{d}\tau\Big)^{\frac{p-2}{2}}
				\int_{t_{1}}^{t_{1}+r}\textbf{E}\|\tau^{1-\alpha}X(\tau)\|^{p}\mathrm{d}\tau.
			\end{align}
		}
		As $r\to 0$ the right-hand-side of the above inequality tends to zero. Thus, by taking into \eqref{I1}-\eqref{I5} account $\Psi$ is continuous in $p$th moment on $[0,T]$.

		Next, to prove the global existence and uniqueness of solutions, we will show that the operator $\Psi$ has a unique fixed point. Indeed for any $X,Y\in H$, we have
		\allowdisplaybreaks
		\begin{align*}
			&\sup_{t\in[0,T]}\textbf{E}\|t^{1-\alpha}(\Psi X)(t)-t^{1-\alpha}(\Psi Y)(t)\|^{p}\\
			&\leq 4^{p-1}\sup_{t\in[0,T]} \textbf{E} \| t^{1-\alpha}g(t,(\Psi X)(t))-t^{1-\alpha}g(t,(\Psi Y)(t))\|^{p}\\
			&+4^{p-1}\sup_{t\in[0,T]}\textbf{E}\|\int_{0}^{t}At^{1-\alpha}(t-\tau)^{\alpha-1}E_{\alpha,\alpha}((t-\tau)^{\alpha}A)[g(\tau,X(\tau))-g(\tau,Y(\tau))]\mathrm{d}\tau\|^{p} \\
			&+4^{p-1}\sup_{t\in[0,T]}\textbf{E}\|\int_{0}^{t}t^{1-\alpha}(t-\tau)^{\alpha-1}E_{\alpha,\alpha}((t-\tau)^{\alpha}A)[b(\tau,X(\tau))-b(\tau,Y(\tau))]\mathrm{d}\tau\|^{p} \\
			&+ 4^{p-1}\sup_{t\in[0,T]}\textbf{E} \|\int_{0}^{t}t^{1-\alpha}(t-\tau)^{\alpha-1}E_{\alpha,\alpha}((t-\tau)^{\alpha}A)[\sigma(\tau,X(\tau))-\sigma(\tau,Y(\tau))]\mathrm{d}W(\tau)\|^{p} \\
			&\leq 4^{p-1}L_{g}^{p}\sup_{t\in[0,T]}\textbf{E}\|t^{1-\alpha}(\Psi X)(t)-t^{1-\alpha}(\Psi Y)(t)\|^{p}\\
			&+ 4^{p-1}L_{g}^{p}\|A\|^{p}\Big(\sup_{t\in [0,T]}\int_{0}^{t}\left(t^{1-\alpha} (t-\tau)^{\alpha-1}\tau^{\alpha-1}\|E_{\alpha,\alpha}((t-\tau)^{\alpha}A)\|\right) ^{\frac{p}{p-1}}\mathrm{d}\tau\Big)^{p-1}\\
			&\times\sup_{t\in [0,T]}\textbf{E}\|t^{1-\alpha}X(t)-t^{1-\alpha}Y(t)\|^{p}\\
			&+4^{p-1}L_{b}^{p}\Big(\sup_{t\in [0,T]}\int_{0}^{t}\left(t^{1-\alpha} (t-\tau)^{\alpha-1}\tau^{\alpha-1}\|E_{\alpha,\alpha}((t-\tau)^{\alpha}A)\|\right)^{\frac{p}{p-1}}\mathrm{d}\tau\Big)^{p-1}\\
			&\times \sup_{t\in [0,T]}\textbf{E}\|t^{1-\alpha}X(t)-t^{1-\alpha}Y(t)\|^{p}\\
			&+4^{p-1}C_{p}\sup_{t\in [0,T]}\textbf{E}\Big[ \int_{0}^{t}\|t^{1-\alpha}(t-\tau)^{\alpha-1}\tau^{\alpha-1}E_{\alpha,\alpha}((t-\tau)^{\alpha}A)[\tau^{1-\alpha}\sigma(\tau,X(\tau))-\tau^{1-\alpha}\sigma(\tau,Y(\tau))]\|_{X}^{2}\mathrm{d} \tau\Big]^{\frac{p}{2}}\\
			&\leq 4^{p-1}L_{g}^{p}\|\Psi X-\Psi Y\|_{H}^{p}\\
			&+4^{p-1}L_{g}^{p}\|A\|^{p}M^{p}\Big(\sup_{t\in [0,T]}\int_{0}^{t}\left(t^{1-\alpha} (t-\tau)^{\alpha-1}\tau^{\alpha-1}\|\right) ^{\frac{p}{p-1}}\mathrm{d}\tau\Big)^{p-1}\|X-Y\|_{H}^{p}\\
			&+4^{p-1}L_{b}^{p}M^{p}\Big(\sup_{t\in [0,T]}\int_{0}^{t}\left(t^{1-\alpha} (t-\tau)^{\alpha-1}\tau^{\alpha-1}\right) ^{\frac{p}{p-1}}\mathrm{d}\tau\Big)^{p-1}\|X-Y\|_{H}^{p}\\
			&+4^{p-1}C_{p}L_{\sigma}^{p}M^{p}\Big(\sup_{t\in [0,T]}\int_{0}^{t}t^{2-2\alpha} (t-\tau)^{2\alpha-2}\tau^{2\alpha-2}\mathrm{d}\tau\Big)^{p/2}\|X-Y\|_{H}^{p}\\
			&\leq 4^{p-1}L_{g}^{p}\|\Psi X-\Psi Y\|_{H}^{p}+4^{p-1}L_{g}^{p}\|A\|^{p}M^{p}\left(B\left(\frac{p\alpha-1}{p-1},\frac{p\alpha-1}{p-1}\right)\right)^{p-1} T^{p\alpha-1}\|X-Y\|_{H}^{p}\\
			&+4^{p-1}L_{b}^{p}M^{p}\left(B\left(\frac{p\alpha-1}{p-1},\frac{p\alpha-1}{p-1}\right)\right)^{p-1} T^{p\alpha-1}\|X-Y\|_{H}^{p}\\
			&+4^{p-1}C_{p}L_{\sigma}^{p}M^{p}T^{p(\alpha-1)+\frac{p}{2}}\left( B(2\alpha-1,2\alpha-1)\right)^{p/2}\|X-Y\|_{H}^{p}.
		\end{align*}
		Thus,
		\begin{align*}
			\|\Psi X-\Psi Y\|_{H}^{p}\leq \frac{\Theta}{1-4^{p-1}L^{p}_{g}}\|X-Y\|_{H}^{p},
		\end{align*}
		where
		\begin{align*}
			\Theta&\coloneqq 4^{p-1}\Big[ L_{g}^{p}\|A\|^{p}M^{p}\left(B\left(\frac{p\alpha-1}{p-1},\frac{p\alpha-1}{p-1}\right)\right)^{p-1} T^{p\alpha-1}\\
			&+L_{b}^{p}M^{p}\left(B\left(\frac{p\alpha-1}{p-1},\frac{p\alpha-1}{p-1}\right)\right)^{p-1}T^{p\alpha-1}\\
			&+C_{p}L_{\sigma}^{p}M^{p}T^{p(\alpha-1)+\frac{p}{2}}\left( B(2\alpha-1,2\alpha-1)\right)^{p/2}\Big].
		\end{align*}.
		
		Therefore, by Assumption \ref{A3}, $\Psi$ is a contraction mapping and hence there exists a unique fixed point, which is a mild solution of \eqref{fstoc} on $[0,T]$.
	\end{proof}
	\section{Asymptotic stability analysis of Riemann-Liouville fractional neutral differential equations}\label{asymptotic}
	We assume that   {$\mathscr{D}$ is a space of all $\mathbb{R}^{n}$-valued $\mathbb{F}_{t}$-adapted process $\varphi$ satisfying the following norm defined by
		\begin{equation*}
			\|\varphi\|_{\mathscr{D}}^{p}=\sup_{t\geq 0}\textbf{E}\|\varphi(t)\|^{p}.
		\end{equation*}
		where $\varphi$ is continuous in $t$ such that $\textbf{E}\|\varphi(t)\|^{p}\to 0$ as $t\to 0$.
		It is clear that $\mathscr{D}$ is a Banach space with respect to the norm $\|\cdot \|_{\mathscr{D}}$.
		\begin{theorem}\label{stabilitytheorem}
			Let $p\geq 2$ be an integer and suppose that Assumptions \ref{A0}- \ref{A3} hold. Then, the mild solution of \eqref{fstoc} is asymptotically stable in the $p$th moment.
		\end{theorem}
		\begin{proof}
			First, we show that $\Psi(\mathscr{D})\subset \mathscr{D}$. Let $X\in \mathscr{D}$. We have
			\begin{align}\label{PsiB}
				\textbf{E}\|(\Psi X)(t)\|^{p}&\leq 6^{p-1}\textbf{E}\|t^{\alpha-1}E_{\alpha,\alpha}(t^{\alpha}A)\rho\|^{p} +6^{p-1} \textbf{E}\|g(t,X(t))\|^{p} \nonumber\\
				&+6^{p-1}\textbf{E}\norm{\int_{0}^{t} A(t-\tau)^{\alpha-1}E_{\alpha,\alpha}((t-\tau)^{\alpha}A) g(\tau,X(\tau))\mathrm{d}\tau}^{p}\nonumber \\
				&+6^{p-1}\textbf{E}\norm{\int_{0}^{t} (t-\tau)^{\alpha-1}E_{\alpha,\alpha}((t-\tau)^{\alpha}A)b(\tau,X(\tau))\mathrm{d}\tau}^{p} \nonumber\\
				&+6^{p-1}\textbf{E}\norm{\int_{0}^{t} (t-\tau)^{\alpha-1}E_{\alpha,\alpha}((t-\tau)^{\alpha}A)\sigma(\tau,X(\tau))\mathrm{d}W(\tau)}^{p}.
			\end{align}
			Now we estimate the terms on the right-hand-side of \eqref{PsiB}. By Assumption \ref{A0} and according to the Lemma \ref{Lem21} (i), first we have
			\begin{align*}
				&6^{p-1}\textbf{E}\|t^{\alpha-1}E_{\alpha,\alpha}(t^{\alpha}A)\rho\|^{p}\to 0, \quad \text{as}\quad  t\to\infty .
			\end{align*}
			For $X(t)\in \mathscr{D}$ and for any $\epsilon>0$, there exists a $t_{1}>0$ such that $\textbf{E}\|X(t)\|^{p} \leq \epsilon$ for $t\geq t_{1}$. Therefore,
			\begin{align*}
				6^{p-1} \textbf{E}\|g(t,X(t))\|^{p}&=6^{p-1} \textbf{E}\| g(t,X(t))-g(t,0)+g(t,0) \|^{p}\\
				&\leq 6^{p-1} \textbf{E}\|g(t,X(t))-g(t,0)\|^{p}\\
				&\leq 6^{p-1}L^{p}_{g} \textbf{E}\|X(t)\|^{p}\to 0, \quad \text{as} \quad t\to\infty.
			\end{align*}
			For the fourth term, by Assumptions \ref{A1} and \ref{A3}:
			\begin{align*}
				&6^{p-1}\textbf{E}\norm{\int_{0}^{t} A(t-\tau)^{\alpha-1}E_{\alpha,\alpha}((t-\tau)^{\alpha}A) g(\tau,X(\tau))\mathrm{d}\tau}^{p}\\
				&\leq 6^{p-1}\textbf{E}\left( \int_{0}^{t} \|A\|(t-\tau)^{\alpha-1}\|E_{\alpha,\alpha}((t-\tau)^{\alpha}A)\| \|g(\tau,X(\tau))\|\mathrm{d}\tau\right) ^{p}\\
				&\leq 6^{p-1}L_{g}^{p}\|A\|^{p}\norm{\int_{0}^{t}\left(  (t-\tau)^{\alpha-1}\|E_{\alpha,\alpha}((t-\tau)^{\alpha}A)\|\right)  ^{\frac{p}{p-1}} \mathrm{d}\tau}^{p-1}\int_{0}^{t}\textbf{E}\|X(\tau)\|^{p}\mathrm{d}\tau \\
				&\leq 6^{p-1}L_{g}^{p}\|A\|^{p}\left( \int_{0}^{t}  \left( \tau^{\alpha-1}\|E_{\alpha,\alpha}(\tau^{\alpha}A)\|\right) ^{\frac{p}{p-1}} \mathrm{d}\tau\right) ^{p-1} \int_{0}^{t}\textbf{E}\|X(\tau)\|^{p}\mathrm{d}\tau \to 0, \quad \text{as} \quad  t\to\infty.
			\end{align*}
			Similarly, applying H\"{o}lder inequality, Assumptions \ref{A1} and \ref{A3}, we also have
			\begin{align*}
				&6^{p-1}\textbf{E}\norm{\int_{0}^{t} (t-\tau)^{\alpha-1}E_{\alpha,\alpha}((t-\tau)^{\alpha}A) b(\tau,X(\tau))\mathrm{d}\tau}^{p}\\
				&6^{p-1}\textbf{E}\left( \int_{0}^{t} (t-\tau)^{\alpha-1}\|E_{\alpha,\alpha}((t-\tau)^{\alpha}A)\| \|b(\tau,X(\tau))\|\mathrm{d}\tau\right)^{p}\\
				&\leq 6^{p-1}L_{b}^{p}\left( \int_{0}^{t}  \left( \tau^{\alpha-1}\|E_{\alpha,\alpha}(\tau^{\alpha}A)\|\right) ^{\frac{p}{p-1}} \mathrm{d}\tau\right) ^{p-1} \int_{0}^{t}\textbf{E}\|X(\tau)\|^{p}\mathrm{d}\tau \to 0, \quad \text{as} \quad  t\to\infty.
			\end{align*}
			For the sixth term of \eqref{PsiB}, we arrive at
			\begin{align*}
				&6^{p-1}\textbf{E}\norm{\int_{0}^{t} (t-\tau)^{\alpha-1}E_{\alpha,\alpha}((t-\tau)^{\alpha}A) \sigma(\tau,X(\tau))\mathrm{d}W(\tau)}^{p}\\
				&\leq 6^{p-1}C_{p}\textbf{E}\left( \int_{0}^{t} (t-\tau)^{2\alpha-2}\|E_{\alpha,\alpha}((t-\tau)^{\alpha}A)\|^{2}\|\sigma(\tau,X(\tau))\|^{2}\mathrm{d}\tau\right) ^{p/2}\\
				&\leq 6^{p-1}C_{p}L_{\sigma}^{p}\left( \int_{0}^{t}\left[  (t-\tau)^{2\alpha-2}\|E_{\alpha,\alpha}((t-\tau)^{\alpha}A)\|^{2}\right]^{\frac{p}{p-2}} \mathrm{d}\tau\right)^{\frac{p-2}{2}}\int_{0}^{t}\textbf{E}\|X(\tau)\|^{p}\mathrm{d}\tau\\
				&\leq 6^{p-1}C_{p}L_{\sigma}^{p}\left( \int_{0}^{t}\left[  \tau^{2\alpha-2}\|E_{\alpha,\alpha}(\tau^{\alpha}A)\|^{2}\right]^{\frac{p}{p-2}} \mathrm{d}\tau\right)^{\frac{p-2}{2}}\int_{0}^{t}\textbf{E}\|X(\tau)\|^{p}\mathrm{d}\tau \to 0, \quad \text{as} \quad  t\to\infty.
			\end{align*}

			Now we are in a position to prove asymptotic stability of the mild solution of \eqref{fstoc}. To do so, as the first step, we have to prove the stability in $p$th moment. Let $\epsilon>0$ be given and choose $\delta>0$ such that $\delta<\epsilon$ satisfies
			
			\begin{align*}
				6^{p-1}M^{p}T^{p(\alpha-1)}\delta&+6^{p-1}\Big[L_{g}^{p}+ L^{p}_{g}\|A\|^{p}M^{p}\left(\frac{p-1}{p\alpha-1}\right)^{p-1} T^{p\alpha-1}\\
				&+L_{b}^{p}M^{p}\left(\frac{p-1}{p\alpha-1}\right)^{p-1} T^{p\alpha-1}+C_{p}L_{\sigma}^{p}M^{p}\left( \frac{T^{2\alpha-1}}{2\alpha-1}\right)^{p/2}\Big]\epsilon <\epsilon
			\end{align*}
			If $X(t)=X(t,\rho)$ is mild solution of \eqref{fstoc} with $\|\rho\|^{p}< \delta$, then $(\Psi X)(t)=X(t)$ satisfies $\textbf{E}\|X(t)\|^{p}< \epsilon$ for every $t\geq 0$. If there exists $\hat{t}$ such that $\textbf{E}\|X(\hat{t})\|^{p}=\epsilon$ and $\textbf{E}\|X(\tau)\|^{p}<\epsilon$ for $\tau\in [0,\hat{t}]$. Then,
			
			\begin{align*}
				\textbf{E}\|X(\hat{t})\|^{p}&\leq 6^{p-1}M^{p}T^{p(\alpha-1)}\delta\\ &+6^{p-1}\Big[L_{g}^{p}+ L^{p}_{g}\|A\|^{p}M^{p}\left(\frac{p-1}{p\alpha-1}\right)^{p-1} T^{p\alpha-1}\\
				&+L_{b}^{p}M^{p}\left(\frac{p-1}{p\alpha-1}\right)^{p-1} T^{p\alpha-1}+C_{p}L_{\sigma}^{p}M^{p}\left( \frac{T^{2\alpha-1}}{2\alpha-1}\right)^{p/2}\Big]\epsilon <\epsilon,
			\end{align*}
			which contradicts to the definition of $\hat{t}$. Therefore, the mild solution of \eqref{fstoc} is asymptotically stable in $p$th moment.
		\end{proof}
		\begin{remark}
			In particular, if we replace Riemann-Liouville fractional derivative with Caputo one when $p=2$, we have already proved the existence and uniqueness of mild solution to Caputo FSNDEs in \cite{arzu}. By the help of \cite{arzu}, Theorem \ref{stabilitytheorem} can be adapted with regard to the appropriate conditions as follows:
			\begin{theorem}
				Suppose that the Assumptions \ref{A1}-\ref{A2} hold. Then, Caputo FSNDEs are mean square asymptotically stable if
				\begin{align*}
					4\Big( L_{g}^{2}\|A\|^{2}M^{2} \frac{T^{2\alpha-1}}{2\alpha-1}+L_{b}^{2}M^{2} \frac{T^{2\alpha-1}}{2\alpha-1}+L_{\sigma}^{2}M^{2} \frac{T^{2\alpha-1}}{2\alpha-1}\Big)<1.
				\end{align*}
			\end{theorem}
		\end{remark}
		
		\section{Conclusion} \label{concl}
		In this paper, we have studied asymptotic stability of solutions in $p$th moment to Riemann-Liouville FSNDEs with fractional order $\alpha\in (\frac{1}{2},1)$. For this class of systems, we can easily derive mean square asymptotic stability under appropriate conditions by assuming $p=2$. Also, we can continue the research in the direction by extending this study to the case that $p=p(x)$ is a variable continuous function which greater or equals than $2$. To do so, stochastic integration in $L_{p(x)}$ space should be investigated and Lemma \ref{prato} should be adapted to the case $p=p(x)$ for further research on asymptotic stability, so we leave this as an open problem for future work.
		
		The main contribution of our results is that we have opened the possibility for a cooperative investigation to solve several issues, for instance, combining the methods of this paper to study the control theory, one may solve controllability of nonlinear case of our problem in finite and infinite dimensional spaces and one can also discuss finite-time stability of semiliniear fractional stochastic differential equations.

		\section*{Acknowledgement} 
		The authors express their gratitude to the editor-in-chief, the editor, and the anonymous reviewer for useful feedback and suggestion.
		

	\end{document}